\documentclass[11pt]{article}

\usepackage{authblk}

\usepackage{hyperref}
\usepackage{amsfonts}
\usepackage{graphicx}
\usepackage{tabularx}
\usepackage{array}
\usepackage[usenames,dvipsnames]{color}
\usepackage{comment}
\usepackage{amsmath}
\usepackage{amsthm}
\usepackage{amssymb}
\usepackage{fullpage}
\usepackage[dvipsnames]{xcolor}
\usepackage{tikz}
\usepackage{listings}
\usepackage{dsfont}
\usepackage{enumerate}
\usepackage{romannum}

\newtheorem{theorem}{Theorem}[section]

\newtheorem{conjecture}[theorem]{Conjecture}
\newtheorem{problem}[theorem]{Problem}
\newtheorem{question}[theorem]{Question}

\newtheorem{lemma}[theorem]{Lemma}
\newtheorem{corollary}[theorem]{Corollary}
\theoremstyle{definition}

\def\epsilon{\varepsilon}

\DeclareMathOperator{\chio}{\chi_o}
\DeclareMathOperator{\mad}{\text{mad}}
\DeclareMathOperator{\ex}{\text{Ex}}
\DeclareMathOperator{\coe}{6\Big(1+\frac{9d^2}{4}\Big)^{\frac{1}{2}}d}
\DeclareMathOperator{\bound}{\Big( 3d-1, \coe\sqrt{n}\Big)}

\title{MAD Phase Transitions in the Oriented Chromatic Number}
\author[1]{Alexander Clow\,\thanks{Supported by the Natural Sciences and Engineering Research Council of Canada (NSERC) through PGS D-601066-2025}}
\date{}

\begin{document}
\pagenumbering{arabic}

\maketitle
\begin{abstract}
    For an oriented graph $G$ the oriented chromatic number of $G$, written $\chio(G)$,
    is the least integer $t$ such that $G$ has a homomorphism to a tournament on $t$ vertices.
    The oriented chromatic number of a simple graph $H$ is the maximum oriented chromatic number over 
    all orientations of $H$.
    Borodin, Kostochka, Ne{\v{s}}et{\v{r}}il, Raspaud, and Sopena
    proved in 1999
    that for all $\epsilon>0$,
    graphs with maximum average degree less than $4-\epsilon$
    have bounded oriented chromatic number.
    This is in some sense optimal, because $1$-subdivisions of cliques demonstrate that there 
    exists graphs with maximum average degree strictly less than $4$
    and oriented chromatic number $\Omega(\sqrt{n})$.
    We prove that for every positive integer $d$,
    every
    $d$-degenerate graph $G$ with sufficiently large order satisfies 
    $\chio(G) \leq \coe 8^{d}\sqrt{n}$.
    This implies for a fixed $r\geq 4$, 
    the optimal bound for the oriented chromatic number
    of graphs with maximum average degree less than $r$ 
    is $\Theta(\sqrt{n})$.
    This complements a bound of Wood,
    who showed that for all $n$ vertex graphs $\chi_o \leq 2\Delta\sqrt{n-1}$
    where $\Delta$ is the maximum degree.
\end{abstract}

\section{Introduction}

In this paper we consider \emph{oriented graphs}, which are directed graphs without loops or multi-edges.
If $H$ is an oriented graph and $G$ is the underlying simple graph of $H$, then we say \emph{$H$ is an orientation of $G$}.
If $H$ is an orientation of $G$ and we consider a parameter of $H$ that does not depend on orientation, then suppose we are considering this parameter in $G$. 
For example, let the maximum degree of $H$ be the maximum degree of $G$.

For an integer $d$, we say a graph $G$ is \emph{$d$-degenerate} if every subgraph of $G$ has minimum degree at most $d$.
For a graph the maximum average degree of $G$ is the maximum average degree taken over all subgraphs of $G$.
Unless otherwise stated suppose all logarithms are base $2$.
For readers not familiar with standard graph theoretic notation see \cite{West1996}.

Given oriented graphs $G=(V,E)$ and $H = (V',E')$ we say a map $\phi: V \rightarrow V'$ is an \emph{oriented homomorphism} from $G$ to $H$ if for all $(u,v)\in E$, there exists an edge $(\phi(u),\phi(v))\in E'$. 
Given an oriented homomorphism $\phi$ from $G$ to $H$, we say that the partition of $V$ given by $\{\phi^{-1}(z): z \in V'\}$ is an \emph{oriented colouring} of $G$. If $H$ has at most $k$ vertices, then we say $\{\phi^{-1}(z): z \in V'\}$ is an oriented $k$-colouring of $G$, and if $G$ admits an oriented homomorphism to an oriented graph with at most $k$ vertices, then we say $G$ is oriented $k$-colourable. The \emph{oriented chromatic number} of an oriented graph $G$, denoted $\chio(G)$ is the least integer $k$ such that $G$ is oriented $k$-colourable. If $G$ is a simple graph, then the oriented chromatic number of $G$, denoted $\chio(G)$, is the maximum oriented chromatic number of any orientation of $G$.
Since throughout this paper we only consider homomorphisms of oriented graphs, we will simply write \emph{homomorphism} rather than oriented homomorphism 
going forward.

Oriented colouring was first introduced by Courcelle \cite{courcelle1994monadic} in the context of second order logic. In the same year Raspaud and Sopena \cite{raspaud1994good} proved that every planar graph has oriented chromatic number at most $80$. 
This upper bound has been yet to be improved, despite significant effort, see \cite{sopena2016homomorphisms}.
The best lower bound for the oriented chromatic number of planar graphs is $18$  \cite{marshall2015oriented}.
This lower bound came after \cite{marshall2007homomorphism, sopena2002there} each of which increased the lower bound by one,
from the lower bound of $15$. 
It is easy to observe that $15$ colours may be required
as there is a $15$ vertex planar graph where each vertex requires a different colour in any oriented colouring.

Oriented colouring various subclasses of planar graphs has also been studied.
In particular, oriented colouring planar graphs with girth $g>3$ 
has received quite a bit of attention \cite{borodin2007oriented,borodin2005oriented, borodin1999maximum,  marshall2013homomorphism, nevsetvril1997colorings, ochem2008oriented, ochem2014oriented, pinlou2009oriented}.
A principle tool in upper bounding the oriented chromatic number of these graph classes is to apply discharging.
This setting is conducive to discharging because 
Euler's formula implies that the maximum average degree of any planar graph with girth $g> 3$ is strictly less than $4-\epsilon$.
Unfortunately, this style of technique will not work to improve the upper bound of $80$ for all general planar graphs,
as there are graphs with maximum average degree less than $4$
whose oriented chromatic number are unbounded.
More on this to come.

Oriented colouring has also been considered for a variety of other classes of sparse graphs.
Graphs with Euler genus larger than $0$ (planar graphs are Euler genus $0$ graphs)
has been considered by a number of authors \cite{aravind2013forbidden,BRADSHAW2025,clow2026oriented,kostochka1997acyclic}.
The best bounds for the oriented chromatic number in terms of Euler genus
are of the form $O(g\log(g))$ while there exists graphs with oriented chromatic number $\Omega(\frac{g}{\log(g)})$ and Euler genus $g$ \cite{clow2026oriented}.
Graphs with bounded degree and/or bounded degeneracy has seen significant attention \cite{CLOW2026,das2017chromatic, duffy2020colourings, duffy2019oriented,dybizbanski2020oriented,kostochka1997acyclic,sopena1997chromatic,sopena1996note} with the best general bounds sitting at 
$(2\ln2 +o(1))\Delta^2 2^\Delta$ and $(2+o(1))\Delta d2^d$ (if degeneracy $d$ grows sublinearly in maximum degree $\Delta$).
Oriented colouring random bounded degree graphs was considered by Gunderson and Nir \cite{gunderson2022oriented}.
Graphs with bounded treewidth were considered in Sopena in \cite{sopena1997chromatic}.

As this literature review demonstrates, oriented colouring 
graphs with bounded degeneracy has seen a lot of attention.
This is because greedy colouring strategies are possible in oriented colouring, 
although they are much more difficult than in normal colouring.
To demonstrate this we point out that the $1$-subdivision (subdivide each edge exactly once) of a clique $K_t$ requires at least $t$
colours in any oriented colouring.
This is because if $u,v,w$ are vertices such that $uvw$ is a directed path of length $2$ (henceforth call a \emph{$2$-dipath}),
then $u$ and $w$ cannot receive the same colour in an oriented colouring.
To see this suppose that they did receive the same colour, call it blue,
then for whatever colour is given to $v$, call it red, red both points to and is pointed at by blue.
Since the $1$-subdivision of any graph is $2$-degenerate, this example implies the oriented chromatic number can grow like $\Omega(\sqrt{n})$
for general $2$-degenerate graphs. 

This flies in the face of several results for planar graphs with large girth \cite{borodin2007oriented, borodin1999maximum,pinlou2009oriented},
which prove upper bounds on the oriented chromatic number of all graphs with maximum average degree $\frac{12}{5}, \frac{11}{4},3,$ and $ \frac{10}{3}$.
Borodin, Kostochka, Ne{\v{s}}et{\v{r}}il, Raspaud, and Sopena \cite{borodin1999maximum}
clarify the situation by proving that for all $\epsilon>0$ there exists a constant $N_\epsilon$,
such that
graphs with maximum average degree strictly less than $4-\epsilon$
have oriented chromatic number at most $N_\epsilon$.
Considering the oriented chromatic number of general graphs with bounded maximum average degree greater than $4$ 
was not considered by these, or seemingly any other, authors.
It is clear that the maximum average degree less than $4-\epsilon$ case has seen more attention
because of its connection to planar graphs with girth at least $4$.

Significantly, 
considering the oriented chromatic number of denser graphs has been studied.
In response to a question of Erd\H{o}s in 1995
Kostochka, {\L}uczak,  Simonyi, and Sopena \cite{kostochka1997minimum}
proved there are graphs with $O(n\log(n))$ edges and oriented chromatic number $n$.
This was improved by F\"{u}redi, Horak, Pareek, and Zhu \cite{furedi1998minimal}
who showed there exists graphs with $n\log(n)-\frac{3n}{2}$ edges and oriented chromatic number $n$.
The number of edges is optimal up to the choice of lower order term.
Moreover, Wood \cite{wood2007oriented} proved all graphs with minimum degree $\log(n)$
have oriented chromatic number $\Omega(\sqrt{n})$,
showing that large oriented chromatic numbers are not just possible, but required
in graphs with more than linearly many edges.
In the same paper \cite{wood2007oriented} Wood 
proved that all graphs $G$ have $\chi_o(G)\leq 2\Delta\sqrt{n-1}$
where $\Delta$ is the maximum degree of $G$.

Our main contribution 
is to consider the oriented chromatic number of graphs with maximum average degree less than $r$,
where we allow $r\geq 4$.
Since this is essentially equivalent, up to the precise choice of $r$,
to considering graphs with bounded degeneracy
we consider graphs with degeneracy at most $d$ where $d$ is constant.

\begin{theorem}\label{Thm: main}
    Let $d\geq 2$ be an integer.
    For all $n$,
    if $G$ is a $d$-degenerate graph with order $n$,
    then
    \[
    \chio(G) \leq \coe8^{d}\sqrt{n}
    \]
\end{theorem}

\begin{corollary}
    Let $r\geq 4$ and $f(n,r)$ be the maximum oriented 
    chromatic number of a graph with $n$ vertices and maximum average degree less  than $r$.
    If $r$ is constant, then $f(n,r) = \Theta(\sqrt{n})$, where the asymptotics are in $n$.
\end{corollary}

\begin{corollary}\label{Coro: Phase}
    Let $0 < \epsilon \leq \frac{1}{2}$. If $G$ is a graph of order $n$
    and $\chi_o(G)\geq n^{\frac{1}{2}+\epsilon}$,
    then 
    \[
    \mad(G) \geq \frac{\epsilon}{3}\log(n)-\frac{2}{3}\log\log(n) - \frac{\log(13)}{3}.
    \]
\end{corollary}

\begin{proof}
    Fix $\epsilon \in (0,\frac{1}{2}]$
    and let $G$ be a graph with $\chi_o(G)\geq n^{\frac{1}{2}+\epsilon}$.
    Let $d$ be the degeneracy of $G$.
    If $d\leq 1$, then $G$ is a forest, and it is easy to verify $\chi_o(G) \leq 3$
    a contradiction.
    So $d \geq 2$.

    Since $d \geq 2$, Theorem~\ref{Thm: main}
    implies that 
    \begin{align*}
        n^{\frac{1}{2}+\epsilon} & \leq \chi_o(G) \\
        & \leq \coe8^{d}\sqrt{n} \\
        & \leq 13d^2 8^d \sqrt{n}.
    \end{align*}
    Hence, taking logarithms and simplifying implies that
    \[
    \epsilon\log(n) \leq \log(13) + 2\log(d) + 3d
    \]
    If $d\geq \log(n)$, then we have proven the proposition.
    Otherwise $\log(d) \leq \log\log(n)$.
    From here we can solve for $d$ on the right hand side and obtain
    \[
    d\geq \frac{\epsilon}{3}\log(n) - \frac{2}{3} \log\log(n) - \frac{\log(13)}{3}.
    \]
    Hence, there is a subgraph of $G$ with minimum degree at least the right hand side of
    the previous expression.
    This completes the proof.
\end{proof}

Observe that Theorem~\ref{Thm: main} implies that 
for all graphs $\chi_o \leq 2^{O(d)}\sqrt{n}$.
Hence, we have replaced the linear $\Delta$ term in
Wood's \cite{wood2007oriented} bound $2\Delta\sqrt{n-1}$
with an exponential function of degeneracy.
Since degeneracy can be arbitrarily smaller than maximum degree
this can lead to a large advantage.
Also note that we have replaced the $\Delta$ term 
in the upper bounds for max degree and degeneracy \cite{CLOW2026,das2017chromatic}
with a $\sqrt{n}$ term.
So for graphs with maximum degree substantially greater than $\sqrt{n}$
this is again an improvement.

Next, observe that Corollary~\ref{Coro: Phase}
implies that any family of graphs
whose oriented chromatic number  is at least  $n^{\frac{1}{2}+\epsilon}$
has maximum average degree at least $(\frac{\epsilon}{3}-o(1))\log(n)$.
On the other hand F{\"u}redi,  Horak, Pareek, and Zhu \cite{furedi1998minimal}
showed there are graphs with average degree $(2-o(1))\log(n)$ and oriented chromatic number
$n$.
A simple examination of their construction yeilds that this graph in fact has 
maximum average degree $(2-o(1))\log(n)$.
So, the lower bound in Corollary~\ref{Coro: Phase} is tight up to a factor of $12$.

We structure the paper as follows.
Section~\ref{sec: Pre} establishes further definitions and ideas that will be used in the paper.
Next, we prove Theorem~\ref{Thm: main} in Section~\ref{sec: Upper}.
We conclude in Section~\ref{sec: future} with a discussion of future work,
including the outline of several open problems and the statement of a conjecture.

\section{Preliminaries}
\label{sec: Pre}

In this section we provide the definitions and notation used in the rest of the paper.
We assume the reader is familiar with standard notation in graph theory.

A graph (or oriented graph) $G = (V,E)$ is an \emph{oriented clique} if and only if $\chio(G) = |V|$. It is not hard to verify that an oriented graph $G$ is an oriented clique if and only if the directed diameter of $G$ is at most $2$. 

We borrow the following notation and definition from \cite{CLOW2026}.
Given an oriented graph $G$, a vertex $v \in V(G)$, and an ordered vertex set 
$U = \{u_1, \dots, u_k\} \subseteq N(v)$, 
we write $F(U,v,G)$ for the vector in $\{-1,1\}^k$
whose $i^{\text th}$ entry is $1$ if $(v,u_i)$ is an edge of $G$, and whose $i^{\text th}$ entry is $-1$ if $(u_i,v)$ is an edge of $G$.

Given a vertex ordering $\sigma = (v_1,\dots, v_n)$ of an oriented graph $G$, we define $\ex_d(G,\sigma)$ as follows.
Let $G_0 = G$ and suppose $G_i$ is defined.
If every vertex $u$ in $G_i$ has degree at most $2$ or at least $3d$ then let $\ex_d(G) = G_i$.
Otherwise, there exists a vertex $u$ in $G_i$ with $3 \leq \deg(u) \leq  3d-1$, let $u$ be such a vertex with minimum index in $\sigma$.
Define $G_{i+1}$ by deleting $u$,
and for all vertex pairs $v \in N^-(u)$ and $w \in N^+(u)$, then we add a degree $2$ vertex $x$ to $G_{i+1}$ where $vxw$ is a $2$-dipath.
Notice that the degree of a fixed vertex is monotone increasing with $i$.
This might lead to a degree $2$ vertex being added, then having its degree increased be larger than $2$ but less than $3d-1$.
These new vertices are not deleted.

Observe that if the set of vertices with degree between $3$ and $3d-1$ forms an independent set,
then $\ex_d(G)$ will contain only vertices with degree at least $3d$ and vertices with degree at most $2$.
Also notice that in such cases the ordering $\sigma$ does not matter. That is, for distinct orderings $\sigma$ and $\tau$,
it is always true that $\ex_d(G,\sigma)$ is isomorphic to $\ex_d(G,\tau)$.
All cases we consider will be of this type.

Letting $k$ and $t$ be positive integers
we say a graph $T$ is $(k,t)$-comprehensive if for all $U \subset V(T)$ where $|U| = k$ and any ${\vec a} \in \{-1,1\}^k$ there exist at least $t$ vertices $z \in V(T)\setminus U$ such that $F(U,z,T) = {\vec a}$.
It is sometimes convenient to take $t$ as a growing function which can lead to it 
not being an integer.
As a convenience, if $r$ is a real number, then we take 
the statement \emph{$T$ is $(k,r)$-comprehensive} to mean $T$ is $(k,\lceil r\rceil)$-comprehensive.

\section{Constructing Oriented Colourings}
\label{sec: Upper}

\subsection{A Universal Target}

Before proving Theorem~\ref{Thm: main}
we first prove the following related claim.

\begin{theorem}\label{Thm: Critical Upper}
    Let $d\geq 2$ be an integer and let $G$ be an orientation of a $d$-degenerate graph.
    If $G$ has $n$ vertices
    and $T$ is a $\bound$-comprehensive graph, then $G$ admits an oriented homomorphism to $T$.
\end{theorem}

To prove Theorem~\ref{Thm: Critical Upper} we need to establish a series of lemmas.
Some of these are well known.
We include all lemmas and their proofs for completeness.

Throughout this subsection it will be convenient 
to consider smallest counterexamples to Theorem~\ref{Thm: Critical Upper}.
The specific integer $d\geq 2$ at hand is always assumed to be fixed but arbitrary unless otherwise stated,
and one can view $d$ as being fixed at all times.
Given a $d$-degenerate oriented graph $G$ that is clear from context
we let $n$ be the number of vertices in $G$.
When $G$ might not be clear from context we write $n(G)$ for the number of vertices in $G$.
Now, for $d$-degenerate oriented graphs $G$ and $H$ we say $G$ is smaller than $H$ for the purposes of a 
smallest counterexample argument if $n(G) < n(H)$ or if $G$ is a strict subgraph of $H$.

\begin{lemma}\label{Lemma: MAD}
    If $G$ is a $d$-degenerate graph, then $\mad(G) < 2d$.
\end{lemma}

\begin{proof}
    Let $G$ be a $d$-degenerate graph.
    Then $G$ admits a vertex ordering $(v_1,\dots, v_n)$ such that for all $i$,
    $|N(v_i)\cap \{v_1,\dots, v_{i-1}\}|\leq d$.
    Immediately, this gives that 
    \[
    |E(G)| = \sum_{i=1}^n |N(v_i)\cap \{v_1,\dots, v_{i-1}\}| \leq dn-\binom{d+1}{2}.
    \]
    Hence, by the handshaking lemma the average degree of a $d$-degenerate graph is equal to
    \[
    \frac{2|E(G)|}{n} < \frac{2dn}{n} = 2d.
    \]
    Since every subgraph of a $d$-degenerate graph is $d$-degenerate we conclude that $\mad(G) < 2d$.
\end{proof}

\begin{lemma}\label{Lemma: comp->}
    Let $k\geq 2$ and $t\geq 1$.
    If $T$ is a $(k,t)$-comprehensive graph, then $T$ is also $(k-1,2t)$-comprehensive.
\end{lemma}

\begin{proof}
    Suppose that $T$ is a $(k,t)$-comprehensive graph where $k\geq 2$. Let $A = \{v_1,\dots, v_{k-1}\} \subset V(T)$ be fixed but arbitrary and let ${\vec a} = (x_1,\dots, x_{k-1}) \in \{-1,1\}^{k-1}$. Next, let $u \in V(T)\setminus A$ and $B = \{v_1,\dots, v_{k-1},u\}$. As $T$ is $(k,t)$-comprehensive there are at least $t$ vertices $z$ such that $F(B,z,T) = {\vec a}_+$, and at least $t$ vertices $w$ such that $F(B,w,T) = {\vec a}_-$, where ${\vec a}_+ = (x_1,\dots, x_{k-1},1)$ and ${\vec a}_- = (x_1,\dots, x_{k-1},-1)$. For any such $z$ or $w$, $F(A,z,T)=F(A,w,T) = {\vec a}$, so there are at least $2t$ vertices $z$ such that $F(A,z,T)= {\vec a}$. As our choice of $A$ and ${\vec a}$ was arbitrary we conclude that $T$ is $(k-1,2t)$-comprehensive as required.
\end{proof}

\begin{lemma}\label{Lemma: counter-no-oneway}
    If $G$ is a smallest counterexample to Theorem~\ref{Thm: Critical Upper},
    then for all vertices $v \in V(G)$ with degree at most $3d-1$,
    $\deg(v)^+,\deg^-(v)\geq 1$.
\end{lemma}

\begin{proof}
    Suppose $G$ is a smallest counterexample and
    for contradiction suppose $v$ is a vertex with degree at most $3d-1$ and $\deg^-(v) = 0$.
    The case $\deg^+(v) = 0$ follows by a similar argument.
    Since $G$ is a counterexample there is a $\bound$-comprehensive graph $T$ such that
    $G$ has no homomorphism to $T$.
    Let $T$ be such a graph.

    Since $G$ is a smallest counterexample $H = G -v$ is not a counterexample.
    Since $n(H)< n(G)$, we note that $H$ has a homomorphism to $T$.
    Let $\phi$ be such a homomorphism
    and let $A = \{\phi(u): u \in N(v)\}$.

    Notice that $|A|\leq \deg(v)\leq 3d-1$.
    Since $T$ is $(3d-1,1)$-comprehensive there is a vertex $x$ in $T$
    such that $\vec{v} = F(A,x,T)$ where $\vec{v}$ is the all $-1$ vector.
    The reader can easily verify that by letting $\phi(v)=x$, $\phi$ is a homomorphism from $G$ to $T$.
    But this is a contradiction, since $G$ has no homomorphism to $T$, completing the proof.
\end{proof}

\begin{lemma}\label{Lemma: counter-low-neighbours}
    If $G$ is a smallest counterexample to Theorem~\ref{Thm: Critical Upper},
    then for all vertices $v \in V(G)$ with degree at most $3d-1$,
    every neighbour $u$ of $v$ has $\deg(u)\geq 3d+1$.
\end{lemma}

\begin{proof}
    Let $G$ be a smallest counterexample to Theorem~\ref{Thm: Critical Upper},
    and for contradiction suppose that there exists two vertices $u$ and $v$ in $G$ such that
    $u$ and $v$ are adjacent in $G$, and we have $\deg(v)\leq 3d-1$ as well as $\deg(u)\leq 3d$.
    Let $T$ be a fixed but arbitrary $\bound$-comprehensive graph.
    Let $H = G - (u,v)$.
    By the minimality of $G$, $H$ is not a counterexample,
    and trivially $n(H) = n(G)$.
    Thus, $H$ admits a homomorphism $\phi$ to $T$.
    Let $\phi$ be such a homomorphism.

    We define a new map $\psi: V(G) \rightarrow V(T)$.
    For all $w \in V(G)\setminus \{u,v\}$ let $\psi(w)=\phi(w)$.
    Since $\phi$ is a homomorphism from $H$ to $G$, 
    without having defined $\psi(u)$ or $\psi(v)$,
    $\psi$ is a homomorphism from $G-\{u,v\}$ to $T$.

    Let $A = \{a_1,\dots, a_{|A|}\}$ be an ordering of $\{\psi(w): w \in N(u)\setminus \{v\}\}$ and let $B = \{b_1,\dots, b_{|B|}\}$ be an ordering of $\{\psi(z): z \in N(v)\setminus \{u\}\}$.
    We define a vector $\vec{w}$ as follows: let $\vec{w} \in \{-1,1\}^{|A|}$ with $i^{th}$ entry $1$ if $(w,u)\in E$, where $\psi(w) = a_i$, and $i^{th}$ entry $-1$ otherwise.
    We note that for all vertices $w,z \in V(G)\setminus \{u,v\}$ such that 
    $wuz$ is a $2$-dipaths in $G$, $\psi(w)\neq \psi(z)$,
    since $wuz$ is a $2$-dipaths in $H$, and every oriented colouring is a $2$-dipath colouring.
    Thus, $\vec{w}$ is well defined.

    Since $\deg(u)\leq (3d-1)+1$ in $G$, and $u$ and $v$ are adjacent in $G$, $|A|\leq \deg(u) - 1\leq (3d-1)$.
    Given $T$ is $\bound$-comprehensive, $T$ is $(3d-1,3d-1)$-comprehensive.
    Hence, there exists at least $3d-1$ vertices $x$ in $T$ such that $\vec{w} = F(A,x,T)$.
    Since $|B|\leq deg(v)-1 < (3d-1)$ it follows that there exists an $x$ in $T$ with $\vec{w} = F(A,x,T)$ and $x \notin B$.
    Let $\psi(u)=x$ for such an $x$.

    Then define $B' = (b_1,\dots, b_{|B|}, b_{|B|+1})$ where we write $x = b_{|A|+1}$.
    Next, let $\vec{z} \in \{-1,1\}^{|B|+1}$ with $i^{th}$ (for $i \leq |B|+1)$ entry $1$ if $(z,v)\in E$, where $\psi(z) = b_i$, and $i^{th}$ entry $-1$ otherwise.
    As with $\vec{w}$ the vector $\vec{z}$ is well defined.
    Since $T$ is $((3d-1),1)$-comprehensive and $|B'|\leq \deg(v)\leq (3d-1)$ there is a vertex $y$ in $T$ such that 
    $\vec{z} = F(B',y,T)$.
    Letting $\psi(v) = y$ the reader can easily verify $\psi$ is a homomorphism from $G$ to $T$.
    Since $T$ was arbitrary
    this contradicts $G$ being a smallest counterexample.
    This concludes the proof.
\end{proof}

\begin{lemma}\label{Lemma: extendedEquiv}
    Let $T$ be a $(3d-1,1)$-comprehensive graph,
    let $G$ be an oriented graph, and let $\sigma$ be a vertex ordering of $G$.
    If the vertices of $G$ with degree at most $3d-1$ form an independent set,
    and no vertex with degree at most $3d-1$ is a source or a sink,
    then $G$ has a homomorphism to $T$ if and only if 
    $\ex_d(G,\sigma)$ has a homomorphism to $T$.
\end{lemma}

\begin{proof}

    Let $T$ be a $(3d-1,1)$-comprehensive graph,
    and suppose the vertices of $G$ with degree at most $3d-1$ form an independent set,
    and suppose that no vertex with degree at most $3d-1$ is a source or a sink.
    The reader can 
    easily verify that if $u \in V(G)$, and $\deg_G(u) > (3d-1)$,
    then $u \in V(\ex_d(G,\sigma))$.
    This is because each time a low degree vertex $v$ is replaced by degree $2$ vertices,
    each neighbour of $v$ will gain at least one degree $2$ neighbour.
   Notice that since no vertex is a source or sink, the minimum degree of $G$ is at least $2$.
   Next, since the vertices of degree at most $3d-1$ form an independent set 
   every vertex $v$ with $2<\deg_G(v) \leq (3d-1)$ in $G$
    will be removed in the process of constructing $\ex_d(G,\sigma)$,
    since their neighbours are not removed, implying their degrees do not increase,
    forcing them to be eventually deleted and replaced by degree $2$ vertices.
    Finally, we note that each added degree $2$ vertex will only be adjacent to vertices with degree at least $3d$,
    so none of the added degree $2$ vertices will increase their degree and subsequently be deleted. 
    Thus, the choice of $\sigma$ has no impact on the resultant $\ex_d(G,\sigma)$.
    Hence, we write $\ex_d(G)$ for $\ex_d(G,\sigma)$ for the remainder of the proof.

    Suppose $G$ has an oriented homomorphism to $T$.
    Let $\phi$ be such a homomorphism.
    We define a new map $\psi$ from $\ex_d(G)$  to $T$.
    Since $\ex_d(G)$ has minimum degree $2$ we can partition the vertices of $\ex_d(G)$
    into two sets $A$ and $B$ where $v \in A$ if $\deg_{\ex_d(G)}(v)>(3d-1)$ and $v \in B$ if $\deg_{\ex_d(G)}(v)=2$.
    From the argument in the previous paragraph every vertex in $A$ is a vertex in $G$,
    moreover, $\ex_d(G)[A]$ is isomorphic to $G[A]$.
    For all $u \in A$ let $\psi(u) = \phi(u)$.
    Since $\phi$ is a homomorphism from $G$ to $T$, $\psi$ as defined is a homomorphism from 
    $\ex_d(G)[A]$ to $T$.
    
    Now, let $v \in B$. Then, Lemma~\ref{Lemma: counter-no-oneway} and the definition of $\ex_d(G)$
    implies that $\deg_{\ex_d(G)}^-(v)=\deg_{\ex_d(G)}^+(v)=1$.
    Suppose without loss of generality that $uvw$ is a $2$-dipath in $\ex_d(G)$.
    Then $u,w\in A$.
    If $v$ is a vertex in $G$, then $\phi(u)\neq \phi(w)$ since all vertices in $A$ are in $G$.
    Hence, $\psi(u)\neq \psi(w)$.
    If $v$ is not a vertex in $G$, then there is a vertex $x$ in $G$ with $\deg_G(x) \leq (3d-1)$
    such that $uxw$ is a $2$-dipath in $G$. So again $\phi(u)\neq \phi(w)$ implying that $\psi(u)\neq \psi(w)$.
    Since $T$ is $(2,1)$-comprehensive $\psi$ can be extended to $v$, while maintain it being an oriented homomorphism.
    Since $v$ was chosen without loss of generality from $B$, the
    map $\psi$ can be extended to a homomorphism from $\ex_d(G)$ to $T$.

    Now suppose that $\ex_d(G)$ has a homomorphism to $T$.
    Let $\psi$ be such a homomorphism.
    We define a new map $\phi$ from $G$ to $T$.
    Define $A$ and $B$ as before.
    Then, letting $\phi(v)= \psi(v)$ for all $v \in A$, the map
    $\phi$ is a homomorphism from $G[A]$ to $T$.
    Now, let $v \in V(G)\setminus A$.
    Then $\deg_G(v)\leq (3d-1)$
    and for all $u \in N^-(v)$ and $w \in N^+(v)$
    there is a degree $2$ vertex $x$ in $\ex_d(G)$ such
    that $uxw$ is a $2$-dipath in $\ex_d(G)$.
    Hence, $\psi(u)\neq \psi(w)$ implying $\phi(u)\neq \phi(w)$.
    Since $T$ is $((3d-1),1)$-comprehensive this implies that
    $\phi$ can be extended to $v$, while maintain it being an oriented homomorphism.
    Since $v$ was chosen without loss of generality from $V(G)\setminus A$, the
    map $\phi$ can be extended to a homomorphism from $G$ to $T$.
    This completes the proof.
\end{proof}

\begin{lemma}\label{Lemma: The Fun Bit}
    Let $G$ be a smallest counterexample to Theorem~\ref{Thm: Critical Upper}, let $\sigma$ be a vertex ordering of $G$,
    let $A$ be the set of vertices with degree at least $3d$ in $\ex_d(G,\sigma)$, and let $B$ be the set of all vertices of degree at most $2$ in $\ex_d(G)$.
    Then 
    for all vertices $v \in A$, if $v$ has $k\leq 3d-1$ neighbours in $A$, 
    then 
    there are at least
    \[
    (\coe\sqrt{n})2^{-k+(3d-1)}.
    \]
    distinct vertices $w\in A$ such that $N(v)\cap N(w)\cap B \neq \emptyset$.
\end{lemma}

\begin{proof}
    Let $G$ be a smallest counterexample to Theorem~\ref{Thm: Critical Upper}.
    Then, as in the proof of Lemma~\ref{Lemma: extendedEquiv} the choice of $\sigma$ has no impact on the structure of $\ex_d(G,\sigma)$
    so we write $\ex_d(G)$ for $\ex_d(G,\sigma)$.
    Since, $G$ is a smallest counterexample,
    there exists a $\bound$-comprehensive graph $T$ where $G$ has no homomorphism to $T$.
    Let $T$ be such a graph.
    Let $A$ be the set of vertices with degree at least $3d$ in $\ex_d(G)$, and let $B$ be the set of all vertices with degree $2$ in $\ex_d(G)$.
    Given Lemma~\ref{Lemma: counter-no-oneway} this partitions the vertices of $\ex_d(G)$.
    Since $G$ is $d$-degenerate $B\neq \emptyset$.
    Similarly, since $G$ is $d$-degenerate
    Lemma~\ref{Lemma: counter-low-neighbours} implies there exists a vertex with degree at least $3d+1$ in $G$, 
    hence $A\neq \emptyset$.

    Suppose there exists a vertex $v\in A$ with $k\leq 3d-1$ neighbours in $A$
    and there exists less than 
    \[
    (\coe\sqrt{n})2^{-k+(3d-1)}.
    \]
    distinct vertices $w\in A$ such that $N(v)\cap N(w)\cap B \neq \emptyset$.
    We claim that $v$ has no neighbour in $B$.
    For contradiction suppose otherwise.
    Assume then without loss of generality that $N^+(v)\cap B\neq \emptyset$.

    Since $N^+(v)\cap B\neq \emptyset$ there exists a vertex $z$ in $G$
    such that $\deg_G(z) \leq 3d-1$ and $(v,z)$ is an arc.
    Define $H = G-vz$. 
    Notice that since $G$ is a smallest counterexample and $v$ is adjacent to $z$, Lemma~\ref{Lemma: counter-low-neighbours} implies that $\deg_G(v)\geq 3d+1$. Hence, $\deg_H(v)\geq 3d$
    implying that $v$ is a vertex in $\ex_d(H)$.
    So, for all vertices $u$ in $H$ with degree at most $3d-1$, if $w$ is a neighbour of $u$, then $w$ has degree greater than $3d-1$.
    Since $G$ is a smallest counterexample and the number of vertices in $H$ is equal to $n$, $H$ has an oriented homomorphism to $T$.
    Thus, by Lemma~\ref{Lemma: extendedEquiv} $\ex_d(H)$ has an oriented homomorphism to $T$.
    Let $\phi$ be such an homomorphism from $\ex_d(H)$ to $T$.

    We define a new map $\psi: V(\ex_d(G)) \rightarrow V(T)$, which we claim will be a homomorphism.
    Notice that $\ex_d(G) - (N(v)\cap B)$ is an induced subgraph of $\ex_d(H)$.
    For all vertices $x \notin (N(v)\cap B) \cup \{v\}$
    let $\psi(x) = \phi(x)$.
    Since $\ex_d(G) - (N(v)\cap B)$ is an induced subgraph of $\ex_d(H)$
    and since we have not coloured $v$, this is a partial homomorphism from $\ex_d(G)$ to $T$.
    We extend $\psi$ to a full homomorphism as follows.

    Observe that since $\phi$ is a homomorphism from $\ex_d(H)$ to $T$,
    for all $2$-dipaths $uvw$, where $u,w\in A$, $\psi(u)=\phi(u)\neq \phi(w)=\psi(w)$. Hence, 
    letting $C = \{c_1,\dots, c_{|C|}\}$ be an ordering of $\{\psi(w): w \in N(v)\cap A\}$ the vector $\vec{v}$ defined by let $\vec{v} \in \{-1,1\}^{|C|}$ with $i^{th}$ entry $1$ if $(w,v)\in E$, where $\psi(w) = c_i$, and $i^{th}$ entry $-1$ otherwise, is well defined.
    
    Recall that $|C|\leq k\leq 3d-1$ by assumption, and recall every neighbour of $v$ in $B$ is uncoloured.
    Since $T$ is $\bound$-comprehensive,
    Lemma~\ref{Lemma: comp->} implies $T$ is $(k,(\coe\sqrt{n}) 2^{-k+(3d-1)})$-comprehensive.
    Thus, there exists 
    at least $(\coe\sqrt{n})2^{-k+(3d-1)}$ vertices $q$ in $T$ such that $\vec{v} = F(C,q,T)$.
    Since there are less than  
    \[
    (\coe\sqrt{n})2^{-k+(3d-1)}
    \]
     distinct vertices $w\in A$ such that $N(v)\cap N(w)\cap B \neq \emptyset$,
    this implies there exists a vertex $q$ in $T$ with 
    $\vec{v} = F(C,q,T)$
    and $q \neq \psi(w)$ for all vertices $w \in A$ such that $v$ and $w$ have a common neighbour in $B$.
    Let $\psi(v)=q$ for such a $q$.
    
    Now, all remaining uncoloured vertices of $\ex_d(G)$ are elements of $B$ adjacent to $v$.
    By our choice of $\psi(v)$, for all uncoloured $u$, both neighbours of $u$ have distinct colours.
    Since $T$ is $(2,1)$-comprehensive, $\psi$ can trivially be extended to a full homomorphism from $\ex_d(G)$ to $T$.
    By Lemma~\ref{Lemma: extendedEquiv} this implies $G$ has a homomorphism to $T$ a contradiction.
    Suppose then that $v$ has
    no neighbours in $B$.
    But $N(v)\cap B=\emptyset$ implies that $\deg_{\ex_d(G)}(v)=k\leq (3d-1)$ contradicting that $v \in A$, since all vertices in $A$ have degree at least $3d$.
    This completes the proof.
\end{proof}

We are now prepared to prove Theorem~\ref{Thm: Critical Upper}.

\begin{proof}[Proof of Theorem~\ref{Thm: Critical Upper}]

We will prove the theorem by induction of $n$.
If $n\leq 1$, then trivially the result holds.
Suppose then that $n>1$ is the least integer such that
there exists a counterexample $G$ to the theorem with $n(G) = n$.
Let $G$ be such a counterexample.
Let $T$ be a $\bound$-comprehensive graph
such that $G$ has no homomorphism to $T$.
Notice $T$ exists since $G$ is a counterexample.

As in Lemma~\ref{Lemma: extendedEquiv} and Lemma~\ref{Lemma: The Fun Bit}
we write $\ex_d(G)$ for $\ex_d(G,\sigma)$ since the choice of $\sigma$ does not matter.
Let $A$ be the set of vertices in $\ex_d(G)$ with degree greater than $(3d-1)$,
and let $B$ be the set of degree at most $2$ vertices in $\ex_d(G)$.
By Lemma~\ref{Lemma: counter-no-oneway} every vertex in $B$ has degree $2$.
Next, 
by Lemma~\ref{Lemma: The Fun Bit} for every vertex $v \in A$ with $k \leq (3d-1)$ neighbours in 
$A$
there are at least
    \[
    (\coe\sqrt{n})2^{-k+(3d-1)}.
    \]
distinct vertices $w\in A$ such that $N(v)\cap N(w)\cap B \neq \emptyset$.
We note that since $G[A]$ and $\ex_d(G)[A]$ are isomorphic,
and since $G$ is $d$-degenerate there exists a vertex in $A$ with at most $d$ neighbours in $A$.
Let $N = |A|$.

Now define $H$ from $\ex_d(G)$ by deleting vertices from $B$ until there are no twin vertices in $B$.
That is until there are no vertices $u,v \in B$ with the same in-neighbours and out-neighbours.
Trivially, $H$ has an oriented homomorphism to $T$ if and only if $\ex_d(G)$ does.
Hence, $H$ has no oriented homomorphism to $T$.
Going forward we will only consider $H$, notice that $A$ is unchanged.
We let $B'$ denote the subset of $B$ in $H$.

Since $B$ is an independent set in $\ex_d(G)$, $B'$ is an independent set in $H$.
Thus, all vertices in $B'$ are degree $2$, with their neighbours in $A$, and $B'$ contains no twins.
Hence, 
\begin{align*}
    n(H) & \leq N +2\binom{N}{2}\\
    & = N^2
\end{align*}
implying $\sqrt{n(H)} \leq N$.
Similarly, notice that each vertex $v$ of degree at most $3d-1$ which is 
removed from $G$ and replaced by degree $2$ vertices in $\ex_d(G)$, 
adds exactly $\deg^-(v)\deg^+$ degree $2$ vertices.
Elementary calculus implies $\deg^-(v)\deg^+\leq \frac{deg(v)^2}{4} < \frac{9d^2}{4}$. 
Since at most $n$ vertices are deleted and replaced with degree $2$ vertices 
during the construction of $\ex_d(G)$ we observe that
\begin{align*}
    n(H) & \leq |V(\ex_d(G))| \\
    & \leq n + \frac{9d^2}{4}n
\end{align*}
implying that $n(H) \leq \Big(1+\frac{9d^2}{4}\Big)n$.
Observe that since $G$ is $d$-degenerate, $\ex_d(G)$ is $d$-degenerate, implying $H$ is $d$-degenerate, 
since $H$ is a subgraph of $\ex_d(G)$.

Since $H$ is $d$-degenerate Lemma~\ref{Lemma: MAD} implies
$\mad(H) < 2d$. So the average degree of $H[A]$ is strictly less than $2d$.
Let $X$ and $Y$ be a partition of $A$ such that $v \in X$
if and only if $\deg_{H[A]}(v)\geq 3d$.
We observe that
\begin{align*}
    2d &> \frac{\sum_{v \in A} \deg_{H[A]}(v)}{N} \\
    & \geq \frac{3d|X|}{N}
\end{align*}
implies that $|X| < \frac{2N}{3}$. From which it follows that 
\[
|Y| > \Big(1-\frac{2}{3}\Big)N\geq  \frac{1}{3}\sqrt{n(H)}.
\]
Now
by Lemma~\ref{Lemma: The Fun Bit} and the definition of $H$ for all vertices $v \in A$ with $\deg_{H[A]}(v) \leq (3d-1)$,
\[
\deg_H(v) \geq  \deg_{H[A]}(v) + \Big( \coe \sqrt{n}\Big)2^{-\deg_{H[A]}(v)+(3d-1)}.
\]
Observe that this function is monotone decreasing for values of $\deg_{H[A]}$ at most $3d-1$.
From here we consider the average degree of the entire graph $H$,
\begin{align*}
    2d &> \frac{\sum_{v \in V(H)} \deg_{H}(v)}{n(H)} \\
    \\
    & \geq \frac{\sum_{v \in Y} \deg_{H}(v)}{n(H)} \\
    \\
    & \geq \frac{1}{n(H)} \Bigg(\sum_{v\in Y} \Big(\deg_{H[A]}(v) + \Big(\coe\sqrt{n}\Big)2^{-\deg_{H[A]}(v)+(3d-1)}\Big)\Bigg) \\
    \\
    &\geq \frac{|Y|}{n(H)}\Bigg(\coe\sqrt{n} \Bigg)  \\
    \\ 
    &>\frac{1}{3\sqrt{n(H)}}\Bigg(\coe\sqrt{n} \Bigg) \\
    \\
    & \geq \frac{\coe }{3} \Bigg(1+\frac{9d^2}{4}\Bigg)^{-\frac{1}{2}}\\
    \\
    & = 2d
\end{align*}
a contradiction. But we have made no assumption except that $G$ is a smallest counterexample.
So there is no counterexample, completing the proof.
\end{proof}

\subsection{Proof of Chromatic Bound}

With Theorem~\ref{Thm: Critical Upper} proven,
we must now prove there exists a target graphs $T$ of small order.
Combining the existence of a smaller order target graph $T$ with Theorem~\ref{Thm: Critical Upper}
immediately provides a proof of Theorem~\ref{Thm: main}.
To accomplish this we proceed by a basic application of the probabilistic method,
similarly to that which appears in \cite{CLOW2026}.

\begin{lemma}\label{Lemma: Chernoff's Bound}[\cite{motwani1995randomized} Theorem~4.2]
    Let $X$ be a random variable with binomial distribution $Bin(n,p)$ where $0<p<1$ and $0 < \delta \leq 1$. Then, $$\mathbb{P}(X < (1-\delta)\mu) \leq \exp(-\mu\delta^2/2)$$
    where $\mu := \mathbb{E}(X)$.
\end{lemma}

\begin{lemma}\label{Lemma: exists comprehensive}
    Let $k\geq 2$ be a constant integer.
    For $t\geq 14k^2$,
    there exists a $(k,t)$-comprehensive graph with order at most 
    $2t2^k$.
\end{lemma}

\begin{proof}
    Let $T = (V,E)$ be a random orientation of the complete graph on $n$ vertices such that each edge is assigned an orientation uniformly and independently. We will choose the value of $n$ later. Let ${\vec a} \in \{-1,1\}^{k}$ and $U \subset V$ such that $|U| = k$ be fixed but arbitrary. Let $X_{U,{\vec a}}$ be the random variable which counts the number of vertices $z \in V\setminus U$ such that $F(U,z,T) = {\vec a}$.
    One can easily verify that $X_{U,{\vec a}}$ has binomial distribution $Bin((n-k),2^{-k})$.
    Hence,
    \[
    \mathbb{E}(X_{U,{\vec a}}) = (n-k)2^{-k}.
    \]
    Now, letting $\delta >0$ be a a constant to be chosen later, notice that Lemma~\ref{Lemma: Chernoff's Bound} implies
    \begin{align*}
        \mathbb{P}\Big(X_{U,{\vec a}} < (1-\delta)(n-k)2^{-k}\Big) & \leq \exp\Bigg(\frac{-(n-k)2^{-k}\delta^2}{2}\Bigg).
    \end{align*}
    If for all $U$ and all $\vec{a}$, $X_{U,\vec{a}}\geq t$, then $T$ is $(k,t)$-comprehensive.
    Hence, by letting $\delta = 1 - \frac{t}{(n-k)2^{-k}}$ for a fixed but arbitrary $U$ and all $\vec{a}$,
    we can bound the probability $X_{U,\vec{a}} < t$ by 
    \begin{align*}
        \mathbb{P}\Big(X_{U,{\vec a}} < t\Big) & \leq \exp\Bigg(\frac{-(n-k)2^{-k}\delta^2}{2}\Bigg) \\
        & = \exp\Bigg( \frac{-1}{2}\Big((n-k)2^{-k} - 2t + \frac{t^2}{(n-k)2^{-k}}\Big) \Bigg)
    \end{align*}
    Letting $n = 2t2^k$ this gives that
    \[
    \mathbb{P}\Big(X_{U,{\vec a}} < t\Big)  < \exp\Big( -\Big(\frac{1}{4} - o(1)\Big)t \Big)
    \]
    Now applying the union bound over all choices of $U$ and $\vec{a}$ gives,
    \begin{align*}
        \mathbb{P}(\text{$T$ is not $(k,t)$-comprehensive}) & \leq \binom{n}{k}2^k\exp\Big( -\Big(\frac{1}{4} - o(1)\Big)t \Big) \\
        & \leq n^k \exp\Big( -\Big(\frac{1}{4} - o(1)\Big)t \Big) \\
        & = t^k 2^{k(k+1)}\exp\Big( -\Big(\frac{1}{4} - o(1)\Big)t \Big).
    \end{align*}
    Since $k$ is constant and $t$ is growing, for all sufficiently large $t$, the probability that $T$ is 
    $(k,t)$-comprehensive is positive.
    A careful analysis of the inequalities here shows that $t = 14k^2$ is sufficient for all $k\geq 2$.
    This completes the proof.
\end{proof}

\begin{proof}[Proof of Theorem~\ref{Thm: main}]
    Let $G$ be a fixed but arbitrary $n$ vertex $d$-degenerate graph.
    By Lemma~\ref{Lemma: exists comprehensive} 
    if $\coe \sqrt{n} \geq 14(3d-1)^2$, then
    there exists a $\bound$-comprehensive
    graph $T$ with order at most 
    \[
    2\Bigg( \coe \sqrt{n} 2^{3d-1}\Bigg) = \coe8^{d}\sqrt{n}.
    \]
    The reader can verify that if $n\geq 196$, the
    $\coe \sqrt{n} \geq 14(3d-1)^2$
    by maximizing 
    \[
    \Bigg(\frac{14(3d-1)^2}{\coe }\Bigg)^2
    \] 
    over all $d\geq 2$.
    If $n< 196$, the claimed upper bound is greater than $n$,
    hence the result is trivial since $\chi_o(G) \leq n$.
    So suppose $n\geq 196$, hence $T$ exists on the desired number of vertices.
    Now Theorem~\ref{Thm: Critical Upper} implies $G$ has a homomorphism to $T$,
    implying
    \[
    \chi_o(G) \leq |V(T)|\leq \coe8^{d}\sqrt{n}.
    \]
    This completes the proof.
\end{proof}

\section{Future Work}
\label{sec: future}

We conclude with some open problems.
The most obvious problem for future work is to ask if the coefficient in terms of $d$
in Theorem~\ref{Thm: main} can be improved.
Notice that at a variety of points in the argument appearing in Section~\ref{sec: Upper}
non-optimal bounds are taken in order to simplify strings of inequalities.
It is not obvious if a more careful analysis here would lead to substantial improvements 
of the constants as $d$ grows (while remaining a fixed constant).
Substantial here meaning a bound of the form $2^{(1+o(1))d}\sqrt{n}$.
For $d=2$, that is maximum average degree less than $4$, it seems likely that better choices can be made to reduce the constant substantially,
but it is not clear that this sort of optimization will
lead to a upper bound that matches the lower bound provided by $1$-subdivisions.

\begin{question}\label{Q: c_r}
   Let $r\geq 4$ be constant.
   What is the largest constant $c_r$ such that for infinitely many 
   choices of $n$, there exists a $n$ vertex 
   graph $G$ with $\mad(G) < r$ and $\chi_o(G)\geq c_r \sqrt{n}$.
\end{question}

Notice that we state Question~\ref{Q: c_r} here in terms of maximum average degree
rather than degeneracy.
This is because the $2d > 2d$ contradiction in the proof of Theorem~\ref{Thm: Critical Upper}
can be replaced with an $r > r$ contradiction, without any further modifications to the argument.
This would potentially lead to a small refinement in the constant 
in Theorem~\ref{Thm: main} where $r < 2d$.
We do not include this in Theorem~\ref{Thm: main}
because having the entire coefficient in terms of degeneracy
allows us to more easily contrast our result with those appearing in the literature.

It is important to note that Question~\ref{Q: c_r}
is not only a question about upper bounds, but also one about lower bounds.
A scan of the oriented colouring literature reveals that there are far more ways to 
upper bound the oriented chromatic number than lower bound it.
Almost every example of an oriented colouring lower bound, see for instance the lower bound
for graphs with bounded degree \cite{kostochka1997acyclic}, graphs with bounded genus \cite{BRADSHAW2025,clow2026oriented},
graphs with bounded treewidth \cite{sopena1997chromatic},
and graphs with at least $n\log(n)$ edges \cite{wood2007oriented}
come from one of two techniques.
Either you construct a oriented clique in your class,
or you apply a basic counting inequality originally observed in \cite{kostochka1997acyclic}.
The notable exception to this are subclasses of planar graphs (including general planar graphs),
see for instance \cite{marshall2007homomorphism, marshall2015oriented,sopena2002there}.

It would be very interesting to construct $2$-degenerate graphs that require substantially more colours
than the $1$-subdivision of cliques.
Particularly, if proving this chromatic lower bound involves 
a novel approach.
Developing such a technique would potentially
be useful in improving lower bounds for well studied classes,
perhaps including planar graphs.

\begin{problem}
    Construct, or existentially prove the existence of, an infinite family of $2$-degenerate 
    graphs such that for all graphs $G$ in this family,
    \[
    \chi_o(G) \geq (1+\epsilon)\sqrt{n}
    \]
    where $n$ is the number of vertices in $G$.
    Of course, once one such family is known, attempt to make $\epsilon$ as large
    as possible.
\end{problem}

It is elementary to observe that a $(k,t)$-comprehensive graph requires at least $t2^k+k$ vertices.
This is because for a fixed set of $k$ vertices, there are $2^k$ possible configurations to another vertex,
and we require at least $t$ vertices that exhibit each of these orientations.
Hence, if one is to use $(d,c\sqrt{n})$-comprehensive graphs as your target for $d$-degenerate graphs
it is impossible to use fewer than $(c+o(1))2^d \sqrt{n}$ colours.
Since there does not seem to be another natural choice of target graph
we conjecture that the exponential term in $d$ here is required.
Again, it would be a major contribution in our estimation 
to prove this using a novel lower bounding technique.

\begin{conjecture}
    There exists a constant $c>0$ such that for all $d\geq 2$,
    there are infinitely many integers $n$, and $d$-degenerate graphs $G$ on $n$ vertices
    with 
    \[
    \chi_o(G) \geq c 2^{d}\sqrt{n}.
    \]
\end{conjecture}

\bibliographystyle{abbrv}
\bibliography{Bib}

\end{document}